\documentclass[12pt]{article}

\usepackage{amsmath,amssymb,amscd,amsfonts,amsthm}
\usepackage{hyperref}
\usepackage{float}
\usepackage{fullpage}
\usepackage{enumerate}

\newtheorem{theorem}{Theorem}

\newtheorem{proposition}{Proposition}
\newtheorem{lemma}{Lemma}

\newtheorem{definition}{Definition}

\title {On Submodular Search and Machine Scheduling}

\author{
Robbert Fokkink \\
Department of Applied Mathematics\\
TU Delft\\
Mekelweg 4\\
2628 CD Delft\\
Netherlands \\
\href{mailto:r.j.fokkink@tudelft.nl}{\tt r.j.fokkink@tudelft.nl} \\
\and
Thomas Lidbetter\\
Department of Management Science\\ and Information Systems\\
Rutgers Business School\\
Newark, New Jersey\\
USA\\
\href{mailto:tlidbetter@business.rutgers.edu}{\tt tlidbetter@business.rutgers.edu} \\
\and
L\'aszl\'o A. V\'egh\\
Department of Mathematics\\
London School of Economics\\
London WC2A 2AE\\
United Kingdom\\
\href{mailto:L.Vegh@lse.ac.uk}{\tt l.vegh@lse.ac.uk} \\
}

\begin{document}
\maketitle

\begin{abstract}
Suppose some objects are hidden in a finite set $S$ of hiding places which must be examined one-by-one.  The cost of searching subsets of $S$ is given by a submodular function and the probability that all objects are contained in a subset is given by a supermodular function. We seek an ordering of $S$ that finds all the objects in minimal expected cost. This problem is NP-hard and we give an efficient combinatorial $2$-approximation algorithm, generalizing analogous results in scheduling theory. We also give a new scheduling application $1|prec|\sum w_A h(C_A)$, where a set of jobs must be ordered subject to precedence constraints to minimize the weighted sum of some concave function $h$ of the completion times of {\em subsets} of jobs. We go on to give better approximations for submodular functions with low {\em total curvature} and we give a full solution when the problem is what we call {\em series-parallel decomposable}. Next, we consider a zero-sum game between a cost-maximizing Hider and a cost-minimizing Searcher. We prove that the equilibrium mixed strategies for the Hider are in the base polyhedron of the cost function, suitably scaled, and we solve the game in the series-parallel decomposable case, giving approximately optimal strategies in other cases.
\end{abstract}

\noindent

\section{Introduction}
\label{sec:intro}
Consider a search problem with a finite, non-empty set $S$ of hiding locations, a {\em cost function} $f\colon 2^S\to \mathbb [0,\infty)$ and a {\em weight function} $g\colon 2^S\to \mathbb [0,\infty)$. An ordering, or {\em search} $\pi$ of $S$ must be chosen. For a given ordering $\pi$ and an element $j$ of $S$, we denote by $S_j=S_j^\pi$ the union of $j$ and all the locations that precede $j$ in the ordering $\pi$. The {\em search cost} of $j$ under $\pi$ is $f(S_j^\pi)$. 
We assume that a Hider has hidden some objects in these locations such that if they are searched according to the ordering $\pi$, the probability that all the objects are in $S_j^\pi$ is $g(S_j^\pi)$. 

We study two variants of the problem. In the \emph{optimization setting}, the Searcher knows the probability distribution used by the Hider, that is, she has oracle access to the function $g$.
In the \emph{game setting}, the objects are adversarially hidden, and thus we consider a two-person zero sum game between the Searcher and the Hider. 
In this paper, we restrict our attention to cases when $f$ is submodular and non-decreasing and $g$ is supermodular and non-decreasing.


\paragraph{The optimization setting} The Searcher can minimize her expected cost by finding an ordering $\pi$ that minimizes the expected search cost with respect to $f$ and $g$, which we write as 
\[ 
c(\pi) = \sum_{j=1}^n \mathbf (g(S_j^\pi)-g(S_j^\pi - j)) f(S_j^\pi). 
\]
We call this problem the {\em submodular search problem}, and if $\pi$ minimizes $c(\pi)$, we say $\pi$ is {\em optimal}.
The equivalent {\em submodular ordering problem} was introduced by Pisaruk~\cite{Pisaruk92}. Pisaruk showed that in the worst case the problem takes exponential time, and he gave a 2-approximation algorithm.
Our first main result, proved in Section~\ref{sec:problem}, provides a simpler and more direct 2-approximation algorithm. Our key new insight is extending {\em Smith's rule}~\cite{Smith} for optimal scheduling to this setting (Theorem~\ref{thm1}). This implies that {\em any} optimal search respects a generalized version of a {\em Sidney decomposition}~\cite{Sidney}; furthermore, any search that respects this decomposition is a $2$-approximation for an optimal search.

We give stronger approximation guarantees for special classes of functions. In Subsection~\ref{sec:curv} we show that our algorithm performs well when the functions $f$ and $g$ are close to being modular.
In particular, we show that the algorithm performs well for low values of the {\em total curvature} of $f$ and $g^\#$ (roughly speaking, the extent they differ from modular functions).
Here, $g^\#(A)=g(S)-g(\bar{A})$ is the dual function of $g$ and $\bar{A}$ denotes the complement of $A$.
Later, in Subsection~\ref{sec:reducible}, we introduce the concept of a {\em series-parallel decomposability}, and show how to find an optimal search if the problem is series-parallel decomposable. 

\paragraph{The game setting} We then restrict our attention to the case when $g$ is modular. This corresponds to the case of hiding a single object at one of the locations according to a probability distribution $\mathbf x \in [0,1]^S$. Thus, $g(A) = \mathbf x(A):=\sum_{j \in A} x_j$. We consider the finite  zero-sum game between a Searcher and a cost maximizing Hider, introduced in~\cite{FKR}.  
A pure strategy for the Searcher is a permutation $\pi$ of $S$ and a pure strategy for the Hider is an element $j \in S$; the payoff is $f(S_j^\pi)$. We call this the {\em submodular search game}. Since the strategy sets are finite, the game has a value and optimal mixed strategies. However the size of the Searcher's strategy set complicates the problem of computing these optimal strategies. This is a {\em search game with an immobile Hider in discrete locations}, which is a type of game that has been well studied, see~\cite{AlpernGal,Gal80,Gal}. It is customary to study such games on graphs, and the cost is given by the time taken for the Searcher to reach the Hider's location from a giving starting point. The alternative approach in our paper is to ignore the graph and focus on the cost function. 

We analyze the submodular search game in Section~\ref{sec:game}, showing that every optimal Hider strategy lies in the {\em base polyhedron} of the scaled cost function $\frac 1 {f(S)} f$, and that any such strategy approximates an optimal strategy by a factor of $2$. We go on to give $\left( \frac 1{1-\kappa_f} \right)$-approximate strategies, where $\kappa_f$ is the total curvature of $f$ (defined precisely in Subsection~\ref{sec:curv}). Finally, we define a notion of series-parallel decomposability for the submodular search game and give a solution in this case.

We do not know the computational complexity of finding equilibrium strategies in the game, and we leave this as an open problem.

\subsection{Motivation, examples, and previous work} \label{sec:ex}
Here, we present a wide range of examples of submodular search in the context of search games and scheduling.  Several further applications are described in~\cite{Pisaruk92}.

\paragraph{Modular search}
In order to give some intuition for the models, we first consider the submodular search problem in the case that $f$ and $g$ are both modular (call this the {\em modular search problem}). In this case, for subsets $A \subset S$, we can write $g(A)=\mathbf x(A)$ and $f(A)=\mathbf c(A)$, for some vectors $\mathbf x, \mathbf c \in \mathbb R^S$. (We note that we are using the symbol $\subset$ to indicate non-strict set inclusion.) The modular search problem was considered by Bellman~\cite{Bellman} (Chapter III, Exercise 3, p.90), and the solution is easily shown to be that $S$ should be searched in non-increasing order of the indices $x_j/c_j$. Blackwell (reported in~\cite{Matula}) considered the more general problem in which each location has an {\em overlook probability}, that is the probability that when a location containing the Hider is inspected, the Hider is not found. An alternative route to the solution of this more complicated problem can be found using Gittins indices for multiarmed bandit processes~\cite{Gittins}. Two different solutions to a game-theoretic version of the modular search problem can be found in more recent work~\cite{AlpernLidbetter,Lidbetter}.

\paragraph{Smith's rule}
The modular search problem is equivalent to a single machine scheduling problem considered in~\cite{Smith}, in which $S$ is a set of jobs, $p_j$ is the {\em processing time} and $w_j$ is the {\em weight} of job $j$. For a given ordering $\pi$ of the jobs, the {\em completion time} $C_j$ of a job $j$ is the sum of its own processing time and the processing times of all jobs that precede it in $\pi$. The objective is to order the jobs so as to minimize the sum $\sum_j w_j C_j$ of the weighted completion times of the jobs. This problem is usually denoted $1|| \sum w_jC_j$, and by writing $p_j = c_j$ and $x_j=w_j$, it clearly fits into the framework of the modular search problem. The solution that the jobs should be completed in non-increasing order of the indices $w_j/p_j$ is known as {\em Smith's rule}. Theorem~\ref{thm1} of this paper is a generalization of Smith's rule, and says that any optimal search in the submodular search problem must begin with a subset $A$ that maximizes $g(A)/f(A)$.

Smith's rule has also appeared in a different guise in the field of reliability theory. In particular, \cite{Gluss} and \cite{Mitten} consider a least cost fault detection problem in which $n$ tests can be performed, each of which has a given cost and a given probability of detecting a fault. The object is to order the tests so as to minimize the expected cost of detecting the fault.

\paragraph{Search on graphs with a single object}
Now consider a generalization of the modular search problem that takes place on a graph on vertex set $S\cup\{r\}$. The Searcher is initially located at $r$ and the Hider is in $S$ according 
to the probability distribution $\mathbf x \in [0,1]^S$. Each edge of the graph has a cost. An {\em expanding search} of the graph is a sequence of edges, the first of which is incident to $r$, while each other is adjacent to some previously chosen edge. For a particular expanding search, the {\em search cost} of a vertex $j$ is the sum of the costs of each of the edges chosen up to and including the first edge that is incident to $j$, and the object is to find an expanding search that minimizes the expected search cost. This problem, which we call the {\em expanding search problem} was introduced in~\cite{AlpernLidbetter}. This paper also considered a game theoretic version of the problem, which we shall refer to the {\em expanding search game}, in which an adversary chooses a worst-case distribution $\mathbf x$. The expanding search paradigm is motivated by scenarios in which there is negligible cost to resume searching from some previously reached point of the search space, for example, when mining for coal. See~\cite{AlpernLidbetter} for further motivations of expanding search. 

Consider the expanding search problem on a tree with root $r$. For a subset $A$ of non-root vertices, let $f(A)$ be the sum of the costs of all the edges in the minimum cardinality subtree containing $A \cup \{r\}$ and let $g(A)=\mathbf x (A)$. Then $f$ is non-decreasing and submodular, $g$ is non-decreasing and modular, and the expanding search problem is equivalent to the submodular search problem for this $f$ and $g$; the expanding search game is equivalent to the submodular search game. In fact, the problem is series-parallel decomposable, so solutions of both follow immediately from this work.

\paragraph{Single machine scheduling with precedence constraints}
Both the expanding search problem and the expanding search game on a tree were solved in~\cite{AlpernLidbetter}, but in fact the expanding search problem is a special case of the single machine scheduling problem $1|prec|\sum w_j C_j$ (see~\cite{Lawler}, for example). This scheduling problem is a generalization of $1||\sum w_j C_j$ for which the ordering of the jobs $S$ must respect some precedence constraints given by a partial order $\prec$ on $S$, so that a job cannot be processed until all the jobs that precede it in the ordering have been completed. Sidney~\cite{Sidney} generalized Smith's rule, showing that an optimal schedule must begin with an {\em initial set} $A$ of jobs that maximizes the ratio $w(A)/p(A)$ (where $A$ is an initial set if for each job $j \in A$, all jobs preceding $j$ in the precedence ordering are also in $A$). Applying this principle repeatedly to the remaining jobs in $\bar{A}$, this gives rise to what became known as a {\em Sidney decomposition} $S=A_1\cup \ldots A_k$, where if $i <j$, all jobs in $A_i$ must be scheduled before all jobs in $A_j$. 

One usually depicts the partial order on the jobs by a Hasse diagram, which is a directed acyclic graph with vertex set $S$ and edges $(s,t)$ if $s\prec t$ and $s$ is an immediate predecessor of $t$. In the case that this graph is a tree, Sidney showed that his decomposition theorem could be used to find an optimal schedule (which was rediscovered in the context of the search problem in~\cite{AlpernLidbetter}). It was later shown that an optimal schedule can be found in polynomial time for generalized series-parallel graphs~\cite{Adolphson,Lawler78}, as we explain in Subsection~\ref{sec:scheduling}, and our result in Subsection~\ref{sec:reducible} for series-parallel decomposable problems generalizes this idea.

The connection to the submodular search problem was pointed out in~\cite{Pisaruk92}.
Define the cost $f(A)$ of a subset $A$ of jobs as the sum $p(\widetilde{A})$ of the processing times of all the jobs in the precedence closure $\widetilde{A}$ of $A$, and define $g(A) = \sum_{j \in A} w_j$. Then $f$ is non-decreasing and submodular and $g$ is non-decreasing and modular, and the problem $1|prec|\sum w_j C_j$ is equivalent to the submodular search problem for this $f$ and $g$. 

The problem $1|prec|\sum w_j C_j$ is well known to be $NP$-hard \cite{Garey,Lawler} (which implies that the submodular search problem is NP-hard) and there are many $2$-approximation algorithms \cite{Ambuhl09,Chekuri,Chudak,Hall,Margot,Pisaruk,Schulz}. Almost all $2$-approximations are consistent with a Sidney decomposition, as shown in~\cite{Correa}. In particular, any ordering of the jobs consistent with a Sidney decomposition approximates an optimal schedule by a factor of $2$. It is also known that there is no polynomial time approximation scheme for the problem unless NP-complete problems can be solved in randomized subexponential time \cite{Ambuhl07}. Furthermore, for any $\varepsilon >0$, there is no $(2-\varepsilon)$-approximation to the problem unless a slightly stronger version of the Unique Games Conjecture fails~\cite{BansalKhot}.

\paragraph{Scheduling with more general costs}
We may also consider the generalization of $1|prec|\sum w_j C_j$, denoted $1|prec|\sum w_j h(C_j)$, in which the object is to minimize the weighted sum of some monotonically increasing function $h$ of the completion times of the jobs. This problem was considered recently in~\cite{Schulz-Verschae}, where the authors find an expression in terms of $h$ for the approximation ratio for an arbitrary schedule that is consistent with a Sidney decomposition for the original problem $1|prec|\sum w_j C_j$. They also show that for any concave $h$, this approximate ratio is at most $2$. The concavity of the function $h$ corresponds to the machine benefiting from a learning effect or from a continuous upgrade of its resources. However, the authors also note that an optimal schedule may not follow a Sidney decomposition of this type. For $h$ concave, $1|prec|\sum w_j h(C_j)$ fits into the submodular search framework, taking $f(A)$ to be $h(p(\widetilde{A}))$ for a subset $A$ of jobs, and $g(A) = \sum_{j \in A} w_j$. Thus we find a different $2$-approximation from~\cite{Schulz-Verschae}, and a Sidney decomposition that is necessarily consistent with every optimal schedule. It should also be mentioned that~\cite{Schulz-Verschae} gives $(2+ \varepsilon)$-approximate algorithms for the more general problem of $1|prec|\sum h_j(C_j)$.

For arbitrary functions $h$ nothing is known about the problem $1|prec|\sum w_j h(C_j)$. Indeed, without any restrictions on $h$, it is a difficult to believe anything can be said in general. If there are no precedence constraints and $h(C_j)=C_j^{\beta}, \beta \geq 0$, this is the problem $1||\sum w_j C_j^{\beta}$, as studied in \cite{Bansal}, in which it is shown that the problem of minimizing total weighted completion time plus total energy requirement (see \cite{Durr,Megow}) can be reduced to $1||\sum w_j C_j^{\beta}, \beta \in (0,1)$. We discuss the problem $1||\sum w_j h(C_j)$ further in Subsection~\ref{sec:scheduling}, in which we bound the approximation ratio of our algorithm by a simple expression in terms of $h$.

\paragraph{Expanding search with multiple objects}
We now extend the expanding search problem to the setting where multiple objects are hidden. Consider a graph on vertex set $S\cup\{v\}$, with several  objects hidden inside $S$, so that for a subset $A\subset S$, objects are hidden at each of the vertices in $A$ with probability $q(A)$, where $\sum_{A \subset S} q(A) = 1$. The objective is to find an expanding search to minimize the expected time to find {\em all} the objects. A game theoretic version of this problem was introduced in~\cite{Lidbetter}, but nothing is known about the problem of minimizing the expected time to find multiple objects hidden according to a known distribution. When the graph is a tree, as before we can define $f(A)$ to be the sum of the costs of all the edges in the minimum cardinality subtree containing $A \cup \{r\}$, and this time define $g(A)$ to be $\sum_{B \subset A} q(B)$. Then $g$ is non-decreasing and supermodular. Thus, this is a submodular search problem, and therefore we obtain a $2$-approximation algorithm.

\paragraph{Scheduling with subset weights}
There is an analogous extension to the scheduling problem $1|prec|\sum w_j C_j$. Instead of giving a weight to each job, we give a weight $w_A \ge 0$ to each {\em subset} $A$ of jobs, and the object is to minimize the sum of the weighted completion times $\sum_{A \subset S} w_A C_A$ of the subsets of the jobs, where $C_A$ is the first time that all the jobs in $A$ have been completed. The motivation for this problem is the prospect that completing certain subsets of jobs could have additional utility. Denote this problem $1|prec|\sum w_A C_A$. If the number of non-zero weights $w_A$ is polynomial in $n$, then $1|prec|\sum w_A C_A$ can be reduced to $1|prec|\sum w_j C_j$. Indeed, given an instance of the former problem, for each subset $A$ with positive weight, we can create a dummy job with processing time $0$ and weight $w_A$ that is preceded by all jobs in $A$. The same holds for the further generalization $1|prec|\sum w_A h(C_A)$, where $h$ is a monotone increasing, concave function of the completion times.

If there are a superpolynomial number of non-zero weights, then the problem $1|prec|\sum w_A h(C_A)$ still fits into our framework: as before, take $f(A)=h(p(\widetilde{A}))$ and this time let $g(A)=\sum_{B \subset A} w_B$. Note that this requires the assumption that the values $g(A)$ are given by an oracle.

This problem can also be interpreted in the context of searching a directed acyclic graph (given by the Hasse diagram of the partial order). For each subset $A$ of edges, objects are hidden at each of the edges in $A$ with probability $w(A)$ (where $w(S)$ is normalized to be equal to 1). An edge can be searched only if all the edges preceding it in the precedence ordering have been searched, and the cost of searching an edge corresponding to a job $j$ is equal to the processing time $p_j$. The objective is to minimize the total expected cost of finding all the hidden objects.

The assumption of an oracle could be reasonable if, for example, $k$ objects are hidden uniformly at random on the edges of a directed acyclic graph, so that $w(A)=1/{n \choose k}$ if $|A|=k$ and $w(A)=0$ otherwise. In this case $g$ is given by $g(A) = {|A| \choose k}/{n \choose k}$. Equivalently, in the scheduling setting, equal utility could be derived from completing all subsets of $k$ jobs. 

\paragraph{The minimum linear ordering problem}
The {\em minimum linear ordering problem} was studied in~\cite{Iwata12}. The problem is to find a permutation $\pi$ to minimize the sum $\sum_{j=1}^n f(S_j^\pi)$ for a function $f:2^S \rightarrow \mathbb R^+$. For $f$ monotone increasing and submodular, an algorithm is given that finds a permutation that approximates an optimal one within a factor of $2-2/(n+1)$. This corresponds to the submodular search problem for $g(A)=|A|$ for all $A$. The approach of~\cite{Iwata12} is quite different to ours or to \cite{Pisaruk92}, and is based on rounding the convex programming relaxation based on the Lov\'{a}sz extension. This technique does not seem to extend easily to the more general setting.

\section{The Submodular Search Problem}
\label{sec:problem}

Let $S=\{1,\ldots,n\}$ be a finite set. A function ${f\colon 2^S\to \mathbb R}$ is {\em submodular} if
\[
f(A\cup B)+f(A\cap B)\leq f(A)+f(B)
\]
for all sets $A,B \subset S$. A function $g\colon 2^S\to \mathbb R$ is {\em supermodular} if and only if $\text{-}g$ is submodular.

We consider the submodular search problem, defined in Section~\ref{sec:intro}, with non-decreasing, non-negative submodular cost function $f$ and non-decreasing, non-negative supermodular weight function $g$. Although we often think of $g$ as defining probabilities, it is simpler not to make the assumption that $g(S)=1$. An optimal search remains optimal if we add a constant to $f$, and submodularity is preserved, so we may assume that $f(\emptyset)=0$ (in other words, $f$ is a polymatroid set function). Similarly, we assume that $g(\emptyset)=0$. 
Further we assume that $f(A)>0$ for all $A \neq \emptyset$, since it is clear that sets with zero cost must be searched first, and we assume that $g(A) <g(S)$ for all $A \neq S$, since any $A$ with $g(A)=g(S)$ would be searched first. We denote the expected cost of a search $\pi$ with respect to functions $f$ and $g$ by $c_{f,g}(\pi)$, though we shall usually suppress the subscripts.


A key concept we will use in the paper is that of the {\em search density} (or simply {\em density}) of a set $A \subset S$, which is defined as the ratio of the probability the Hider is located in $A$ and the cost of searching $A$, if $A$ is searched first. Search density is a concept that often appears in the theory of search games (see \cite{AlpernHoward,AlpernLidbetter,AlpernLidbetter2}), and a general principle that arises is that it is best to search regions of higher density first. The corresponding inverse ratio of the processing time to the weight of jobs also arises naturally in scheduling theory, particularly in the well-known Smith's rule \cite{Smith} for minimizing the weighted completion time in single machine scheduling of jobs without precedence constraints. The rule says that the jobs should be executed in non-decreasing order of this ratio. Our $2$-approximation for the submodular search problem relies on a key result that there is an optimal search that begins with a maximum density subset of $S$. Sidney observed this to be the case for the scheduling problem $1|prec|\sum w_j C_j$ in~\cite{Sidney}.

The proof of our result and the resulting $2$-approximation is inspired by the proof of the analogous result in~\cite{Chekuri}, of which this is a generalization. We emphasize that the $2$-approximation found in~\cite{Chekuri} was obtained independently by~\cite{Margot}. We also note that the $2$-approximation result generalizes a similar result from~\cite{FKR}, which says that {\em any} search strategy is a $2$-approximation for the equilibrium search strategy in the submodular search game.

\begin{definition}
	The {\em search density} (or simply {\em density}) of a non-empty subset $A\subset S$ is defined as
	\[
	\rho(A)=\frac{g(A)}{f(A)}.
	\]
	We denote $\max\{\rho(A)\colon A\subset S\}$ by $\rho^*$ and if $\rho(A)=\rho^*$ then we
	say that $A$ has {\em maximum search density}, or simply {\em maximum density}.
	We put $\rho(\emptyset)=\rho^*$.
\end{definition}

Recall that $\mathcal F\subset 2^S$ is a {\em lattice} if $A,B\in\mathcal F$
implies that $A\cup B\in\mathcal F$ and $A\cap B\in\mathcal F$.
A non-empty $A\in\mathcal F$ is an {\em atom} if the only proper subset of $A$ in $\mathcal F$
is the empty set.
Atoms are disjoint and each element of $\mathcal F$ is a union of atoms.

If $f_1$ and $f_2$ are set functions on disjoint sets $S_1$ and $S_2$ then the {\em direct sum} $f_1 \oplus f_2$ of $f_1$ and $f_2$ over $S_1$ and $S_2$ is the set function on $S_1 \cup S_2$ defined by
\[
(f_1 \oplus f_2) (A) = f_1(S_1 \cap A) + f_2(S_2 \cap A).
\]
The restriction of $f$ to a subset $A$ is denoted by $f|_A$, and similarly for $g$.

In the proof of Lemma~\ref{lem:direct}, and later in the proof of Lemma~\ref{lem:sep}, we use the following observations: if $a,c\ge 0$ and $b,d>0$, then $\frac a b\leq \frac c d$ implies 
\begin{enumerate}[(i)]
	\item {}$\frac a b\leq \frac{a+c}{b+d}\leq \frac c d$.  Furthermore, if one of these three inequalities is an equality, then all the inequalities are equalities.\label{eq:frac-1}
	\item {}$(a-c)\frac{d}{c}\le (b-d)$. \label{eq:frac-2}
\end{enumerate}

\begin{lemma}\label{lem:direct}
	Let $\mathcal M$ be the family of subsets of maximum density and let $M$ be the union of all the atoms of $\mathcal M$.
	Then $\mathcal M$ is a lattice and the functions $f|_M$ and $g|_M$ are both direct sums over the atoms.
\end{lemma}
\begin{proof}
If $A,B\in\mathcal M, A \neq B$ then $\rho^*=g(A)/f(A)=g(B)/f(B)$ and
\[
\rho^*=\frac{g(A)+g(B)}{f(A)+f(B)} \leq
\frac{g(A\cup B)+g(A\cap B)}{f(A\cup B)+f(A\cap B)},
\]
by the submodularity of $f$ and the supermodularity of $g$.
This inequality is in fact an equality, since $\rho(A\cup B)$ and $\rho(A\cap B)$ are both bounded above by $\rho^*$. It follows
that both $A\cup B$ and $A\cap B$ have maximum density.
If $A$ and $B$ are atoms then $A\cap B=\emptyset$, and the equality implies that
$f(A)+f(B)=f(A\cup B)$ and $g(A)+g(B)=g(A\cup B)$, so $f|_{A \cup B}$ and $g|_{A \cup B}$ are both direct sums over $A$ and $B$.
Therefore, $\mathcal M$ is a lattice and~$f|_M$ and $g|_M$ are direct sums over the atoms. 
\end{proof}

We now prove that optimal searches must start with a subset of maximum density, generalizing the analogous result for machine scheduling, as first shown in~\cite{Sidney}.

The proof of the theorem relies on the following lemma. For a subset $A$ of $S$ and $s \in A$, we write $d_s g(A)$ for $g(A)-g(A-\{s\})$, for convenience of presentation, so that, for instance, 
\[
c(\pi) = \sum_{j=1}^n d_j g(S_j^\pi) f(S_j^\pi) .
\]
\begin{lemma} \label{lem:perm}
	Let $f:2^S \rightarrow \mathbb R$ be non-decreasing and let $g:2^S \rightarrow \mathbb R$ be supermodular. If $\pi$ and $\pi'$ are two permutations of $S$, then
	\[
	c(\pi) \ge \sum_{j=1}^n d_j g(S_j^{\pi'}) f(S_j^{\pi}) .
	\]
\end{lemma}
\begin{proof}
We prove Lemma~\ref{lem:perm} using an adjacent pairwise interchange argument.	Suppose the element $i \in S$ appears before $h \in S$ in $\pi$, and suppose $\sigma$ and $\tau$ are any two permutations of $S$ that are identical except that in $\sigma$, the element $i$ appears immediately before $h$ and in $\tau$, the element $h$ appears immediately before $i$. In this case we say that $\tau$ can be obtained from $\sigma$ by a {\em down-switch}. Let $k$ be the immediate predecessor of $i$ in $\sigma$ and of $h$ in $\tau$, and let $T=S_k^\tau=S_k^\sigma$. Then
\begin{align}
&\sum_{j=1}^n d_j g(S_j^{\sigma}) f(S_j^{\pi}) - \sum_{j=1}^n d_j g(S_j^{\tau}) f(S_j^{\pi})   = (d_i g(S_i^\sigma) - d_i g(S_i^\tau) ) f(S_i^\pi) + ( d_h g(S_h^\sigma) - d_h g(S_h^\tau) ) f(S_h^\pi)\nonumber \\
&= ( g(T\cup\{i,h\})-g(T\cup\{i\}) - g(T\cup\{h\})+g(T) ) (f(S_h^\pi) - f(S_i^\pi)). \label{eq1}
\end{align}
By the monotonicity of $f$ and the supermodularity of $g$, the left-hand side of (\ref{eq1}) is non negative.

It is easy to see that every permutation $\pi'$ can be derived from $\pi$ by performing a finite number of down-switches, and this proves the lemma. 
\end{proof}

We say that $A$ is an {\em initial segment} of a search strategy $\pi$ if $A=\{\pi(1),\ldots,\pi(|A|)\}$. 

\begin{theorem}\label{thm1}
	Let $M$ be the element of $\mathcal M$ of largest cardinality.
	Then any optimal search $\pi$ has initial segment $M$.
	Furthermore, if $A\in\mathcal M$, then
	there exists an optimal search $\pi'$ such that $A$ is an initial segment.
\end{theorem}

\begin{proof}
Let $A$ be any subset of maximum search density.
Suppose that an optimal search $\pi$ starts by searching sets $B_1, A_1, B_2, A_2,\ldots,B_k,A_k \subset S$ in that order
before searching the rest of $S$, where $A_i\subset A$ and $B_i\subset \bar{A}$ for all $i$, the union $A_1 \cup \ldots \cup A_k$ is equal to $A$, and $B_1$ may be
the empty set. Let $A^j=A_1\cup \ldots \cup A_j$ and similarly ${\text{for }B^j}$.

Define a new search $\pi'$ which starts by searching $A_1,\ldots,A_k, B_1,\ldots, B_k$ before searching the
rest of $S$ in the same order. Within each $A_i$ and $B_i$ the new search follows the same order as $\pi$. For a subset $T$ of $S$, let $\Delta(T)$ be the difference between the terms corresponding to elements of $T$ in $c(\pi)$ and in $c(\pi')$. We will show that $\Delta \equiv \Delta(S) = 0$.

First consider any $s\in A_j$. The difference $\Delta(\{s\})$ is
\begin{align*}
\Delta(\{s\}) &= d_s g(S_s^\pi) f(S_s^\pi)  - d_s g(S_s^{\pi'}) f(S_s^{\pi'}) \\
& = d_s g(S_s^{\pi'})  (f(S_s^\pi) - f(S_s^{\pi'}))  + (d_s g(S_s^\pi) - d_s g(S_s^{\pi'})) f(S_s^\pi) \\
& \ge  d_s g(S_s^{\pi'}) (f(A \cup B^j) - f(A))+   (d_s g(S_s^\pi) - d_s g(S_s^{\pi'})) f(S_s^\pi),
\end{align*}
by the submodularity of $f$ and the monotonicity of $g$. Summing over all $s \in A_j$ gives
\begin{align}
\Delta(A_j) & \ge (g(A^{j})-g(A^{j-1})) (f(A \cup B^j)-f(A))  +  \sum_{s \in A_j}  (d_s g(S_s^\pi) - d_s g(S_s^{\pi'}))  f(S_s^\pi) \nonumber \\
& \ge \frac{1}{\rho^*} (g(A^{j})-g(A^{j-1})) (g(A \cup B^j)-g(A))  +  \sum_{s \in A_j}  (d_s g(S_s^\pi) - d_s g(S_s^{\pi'}))  f(S_s^\pi), \label{eq:Adiff}
\end{align}
The second inequality used the inequality \eqref{eq:frac-2} above, noting that $1/\rho^*=f(A)/g(A)$. Now consider any $t\in B_j$. The difference $\Delta(\{t\})$ is
\begin{align*}
\Delta(\{t\}) & =  d_t g(S_t^\pi) f(S_t^\pi)  - d_t g(S_t^{\pi'}) f(S_t^{\pi'})  \\
& =  d_t g(S_t^{\pi'}) (f(S_t^{\pi}) - f(S_t^{\pi'})) + (d_t g(S_t^\pi) - d_t g(S_t^{\pi'})) f(S_t^\pi)   \\
& \ge  d_t g(S_t^{\pi'}) (f(A^{j-1}) - f(A)) + (d_t g(S_t^\pi) - d_t g(S_t^{\pi'})) f(S_t^\pi) 
\end{align*}
by the submodularity of $f$. Summing over all $t \in B_j$ gives
\begin{align}
\Delta(B_j) & \ge (g(A\cup B^{j}) - g(A\cup B^{j-1})) (f(A^{j-1}) - f(A)) + \sum_{t \in B_j} (d_t g(S_t^\pi) - d_t g(S_t^{\pi'})) f(S_t^\pi) \nonumber \\
& \ge \frac{1}{\rho^*} (g(A\cup B^{j}) - g(A\cup B^{j-1})) (g(A^{j-1}) - g(A)) + \sum_{t \in B_j} (d_t g(S_t^\pi) - d_t g(S_t^{\pi'})) f(S_t^\pi), \label{eq:Bdiff}
\end{align}
again using \eqref{eq:frac-2}. We now sum these estimates on $\Delta(A_j)$ and $\Delta(B_j)$ over all $j$. Adding the two sums in the right-hand sides of~(\ref{eq:Adiff}) and~(\ref{eq:Bdiff}) and summing over $j$, we obtain
\[
\sum_{j=1}^n d_j g(S_j^\pi) f(S_j^\pi)  - \sum_{j=1}^n d_j g(S_j^{\pi'}) f(S_j^{\pi}),
\]
which is non-negative, by Lemma~\ref{lem:perm}. Hence $\Delta$, which is equal to the sum over $j$ of the right-hand sides of~(\ref{eq:Adiff}) and~(\ref{eq:Bdiff}), satisfies
\begin{align*}
\rho^* \Delta & \geq \sum_{j=1}^k \left( (g(A^{j})-g(A^{j-1})) (g(A \cup B^j)-g(A)) + (g(A\cup B^{j}) - g(A\cup B^{j-1})) (g(A^{j-1}) - g(A) \right) \\
& =  \sum_{j \le k} (g(A^j) - g(A^{j-1})) \sum_{i \le j} (g(A \cup B^i) - g(A \cup B^{i-1})) \\
& \quad + \sum_{j \le k} (g(A \cup B^j) - g(A \cup B^{j-1})) \sum_{i \ge j} (g(A^i) - g(A^{i-1})) \\
& =0,
\end{align*}
by swapping the order of summation of one of the double sums.

Therefore the ordering $\pi'$ is optimal. Hence, it must be true that $\Delta=0$ and all inequalities above are equalities.
It follows that $\rho(A\cup B^j)=\rho(A)=\rho^*$ for all $j$, and in particular $\rho(A\cup B)=\rho(A\cup B^k)=\rho^*$.
We have thus established that if $A$ has maximum search density, then it is a subset of an initial segment
$A\cup B$ of maximum density. Therefore, every optimal strategy $\pi$ searches $M$ first.
We have also established that there exists an optimal search that has
$A$ as an initial segment. 
\end{proof}

\subsection{A 2-approximation}
\label{sec:2approx}

Theorem~\ref{thm1} suggests an approach to constructing an optimal strategy, akin to a Sidney decomposition~\cite{Sidney} for machine scheduling. First find a non-empty subset ${A\subset S}$ of maximum density. By Theorem~\ref{thm1} there is an optimal strategy that begins with the elements of $A$. Now consider the subproblem of finding an optimal search of $\bar{A}$ with cost function $f_A$ defined for $B \subset \bar{A}$ by $f_A(B) =
f(A\cup B)-f(A)$ and weight function $g_A$ defined by $g_A(B)=g(A \cup B)-g(A)$. The function $f_A$ is called the {\em contraction} of $f$ by $A$ and is well known to be submodular \cite[page 45]{Fujishige}. Similarly, the contraction $g_A$ is supermodular. It is easy to see that a search of $S$ that begins with the elements of $A$ is optimal only if it defines an optimal search of $\bar{A}$ with cost function $f_A$ and weight function $g_A$. We summarize the observation below. 

\begin{lemma} \label{lem:decomp} Suppose there is an optimal search of $S$ with initial segment $A$. Then an optimal search of $S$ can be found by combining an optimal search of $A$ with respect to cost function $f|_A$ and weight function $g|_A$ with an optimal search of $\bar{A}$ with respect to cost function $f_A$ and weight function $g_A$.
\end{lemma}

We now repeat the process on $\bar{A}$ with cost function $f_A$ and weight function $g_A$, finding a subset of maximum density, and so on. The result is a partition of $S$ into subsets $A=A_1,A_2,\ldots,A_k$ such that there exists an optimal search strategy that respects the ordering of those subsets. This is a generalization of the notion of a Sidney decomposition for optimal scheduling \cite{Sidney}. If each subset $A_j$ is chosen to be the maximal set of maximum density, then Theorem~\ref{thm1} implies that the resulting decomposition must be respected by any optimal search strategy.

We show that in fact, {\em any} search that respects the ordering of such a decomposition $A_1,\ldots,A_k$ described above approximates an optimal search by a factor of $2$, generalizing the analogous result for scheduling that can be found in~\cite{Chekuri} and~\cite{Margot}. We first show that if $S$ itself has maximum density then any search approximates an optimal search by a factor of 2.

\begin{lemma}\label{lem1}
	Suppose that $S$ has maximum search density.
	Then every search strategy has an expected cost in between
	$g(S)f(S)/2$ and $g(S)f(S)$.
\end{lemma}

\begin{proof}
Let $\pi$ be any search, and without loss of generality suppose $\pi(j)=j$, so that $S_j=S_j^{\pi} = \{1,\ldots,j\}$. Write $x_j = g(S_j)-g(S_{j-1}),j=1\ldots,n$, and note that $g(S_j) = \sum_{i \le j} x_i$. Then the expected cost of $\pi$ is
\begin{align*}
c(\pi)&=\sum_j x_j f(S_j) \\
&\geq  \frac 1 {\rho^*} \sum_j x_j g(S_j) \mbox{ (since $\rho(S_j) \le \rho^*$)} \\
&= \frac 1 {\rho^*}  \left( \sum_j x_j ^2+\sum_{i<j} x_i x_j \right) \\
&=\frac 1 {2\rho^* } \left( \left(\sum_j x_j \right)^2+ \sum_j x_j^2 \right) \\
&=\frac 1 {2\rho^* }  \left( g(S) ^2+  \sum_j x_j^2 \right)\\
&\ge\frac {g(S) f(S)}{2}, \nonumber
\end{align*}
since $g(S)/f(S) = \rho^*$. The cost of {\em any} search is at most $g(S)f(S)$.
It follows that if $S$ has maximum search density then $g(S)f(S)/2 \leq c(\pi)\leq g(S)f(S)$. 
\end{proof}

Our $2$-approximation relies on being able to find a maximum density subset efficiently. The problem of maximizing the ratio of a supermodular function to a positive submodular function was considered in~\cite[Section 6]{Iwata}, where it was shown that the problem can be solved in strongly polynomial time. For completeness, we present below a simple version of this algorithm which exploits the fact that $f$ is non-decreasing.
\medskip 

\begin{enumerate}
	\item Set $\lambda = \rho(S)$.
	\item Maximize the supermodular function $g(X)-\lambda f(X)$ over subsets $X \subset S$. Let $A$ be a maximizer.
	\item If $\rho(A)=\rho(S)=\lambda $, return $S$ as a maximum density subset.
	\item Otherwise, set $S=A$ and go back to Step 1.
\end{enumerate}

\medskip
Before we prove the correctness of this algorithm, first note that the total number of iterations is at most $n$, and each iteration involves a minimization of submodular functions, which can be performed in strongly polynomial time, using Schrijver's algorithm,
or the Iwata-Fleischer-Fujishige algorithm~\cite{Fujishige,Schrijver}.

To prove the algorithm does indeed return a maximum density subset, first note that if $A$ maximizes $g(X)-\lambda f(X)$ and $\rho(A)=\rho(S)=\lambda$, then for any set $B \subset S$, we have $g(B) - \lambda f(B) \le g(A) - \lambda f(A)=0$, so $\rho(B) \le \lambda = \rho(S)$, so $S$ has maximum density.

So we just need to show that if $A$ is a maximizer of $g(X)- \lambda f(X)$ then $A$ contains a maximum density subset. Indeed, suppose $B$ has maximum density. Then by the supermodularity of $g- \lambda f$ and the fact that $A$ maximizes $g-\lambda f$, it follows that $g(B) - \lambda f(B) \le g(A \cap B) - \lambda f(A \cap B)$. This can be rewritten as
\[
(\rho(B) - \lambda)f(B) \le (\rho(A \cap B) - \lambda)f(A \cap B).
\]
Since $B$ has maximum density and $f$ is non-decreasing, it follows that $\rho(A \cap B) = \rho(B)$ and $f(A \cap B) = f(B)$, so $A \cap B$ is non-empty and has maximum density.

\begin{theorem}\label{thm2}
	Suppose that the submodular function $f$ and the supermodular function $g$ are given by value  oracles.
	Then there is a $2$-approximation for an optimal search strategy to the submodular search problem that can be computed in strongly polynomial time.
\end{theorem}

\begin{proof}
As discussed above, a subset $A \subset S$ of maximum density can be computed in strongly polynomial time. 
If $A$ is the entire set, then any search is a $2$-approximation
by Lemma~\ref{lem1}. If $A$ is a proper subset,
then there exists an optimal search with initial segment $A$. 

Let $\pi^*$ be an optimal search of $S$, let $\pi_A$ be an optimal search of $A$ with respect to functions $f|_A$ and $g|_A$, and let $\pi_{\overline{A}}$ be an optimal search of $\bar{A}$ with respect to functions $f_A$ and $g_A$. Then
\[
c_{f,g}(\pi^*) = c_{f|_A,g|_A}(\pi_A) + g_A(\bar{A})f(A) + c_{f_A,g_A}(\pi_{\overline{A}}).
\]
By induction, if we have a $2$-approximation for $\pi_A$ and $\pi_{\overline{A}}$ then we have one for $\pi^*$. 
\end{proof}

The algorithm produces a partition $A_1,\ldots,A_k$ of $S$ such that each $A_i$ has maximum density
in the complement of $\cup_{j<i}A_j$. The resulting search strategy $\pi$ orders each $A_i$
in an undetermined manner. The search strategy $\pi$ is fully determined only if each $A_i$ is a singleton.
This only happens in very specific cases, for instance, if $f$ and $g$ are modular.
The maximum density first algorithm then produces an optimal search strategy. As mentioned in the Introduction, this corresponds to Smith's rule~\cite{Smith} for optimal scheduling or the result of Bellman~\cite{Bellman} in the context of search theory.

We note that Pisaruk's algorithm~\cite{Pisaruk92,Pisaruk} also produces a Sidney decomposition of $S$.
The important addition we have made here is Theorem~\ref{thm1}, which implies that {\em every} optimal search follows a Sidney decomposition. Theorem~\ref{thm1} is also important in the next subsection where we give a more refined expression for the approximation ratio of our algorithm.

\subsection{Improved approximation for functions of low curvature}
\label{sec:curv}

Define the {\em dual} $g^\#:2^S \rightarrow \mathbb{R}$ of the set function $g$ by $g^\#(A)=g(S)-g(\overline{A})$ (see~\cite[page 36]{Fujishige}). It is easy to see that $(g^{\#})^{\#}=g$. Also, $g$ is non-decreasing and submodular with $g(\emptyset)=0$ if and only if $g^{\#}$ is non-decreasing and supermodular with $g(\emptyset)=0$.

Observe that for a search $\pi$, we have $c_{f,g}(\pi)=  c_{g^{\#},f^{\#}} (\pi')$, where $\pi'$ is the reverse of $\pi$. Indeed,
\begin{align*} 
c_{f,g}(\pi) &= \sum_{j=1}^n f(S_j^\pi) (g(S_j^\pi) - g(S_j^\pi - j)) \\
&= \sum_{j=1}^n (f^\#(S ) - f^\#(\overline{S_j^\pi })) (g^\#(\overline{S_j^\pi -j}) - g^\#(\overline{S_j^\pi} )) \\
& = \sum_{j=1}^n g^\#(S_j^{\pi'}) (f^\#(S_j^{\pi'}) - f^\#(S_j^{\pi'} - j)) \\
& = c_{f^\#,g^\#}(\pi').
\end{align*}
It follows that $\min_\pi c_{f,g}(\pi) = \min_\pi c_{g^{\#},f^{\#}} (\pi)$, and we will use this {\em duality} later.

We now show that the algorithm of Subsection~\ref{sec:2approx} performs better when the cost function $f$ and the dual function $g^\#$ have {\em total curvature} less than 1.
The {\em total curvature} $\kappa$ of a set function $f$ on $S$ such that $f(\emptyset)=0$ and $f(s)>0$ for all $s \in S$ is 
\[
\kappa=1-\min_{s\in S}\frac{f_{S - s}(s)}{f(s)}=\max_{s \in S} \frac{f(s)+f(S-s)-f(S)}{f(s)}.
\]  
This was first defined in \cite{Conforti}; see also \cite{Vondrak}.
When $f$ is monotone non-decreasing, $\kappa \le 1$. When it is submodular, $\kappa \geq 0$ (with equality if and only if $f$ is modular), and the value $f_X(s)$ decreases as $X\subset S$ increases, but it always exceeds $(1-\kappa)f(s)$.
Note that if $\kappa_{g^\#}$ is the total curvature of $g^\#$ for a supermodular function $g$ with $g(\emptyset)=0$, then $\kappa_{g^\#}={(g(S)-g(s)-g(S-s))}/{(g(S)-g(S-s))}$.

\begin{lemma} \label{lem:curv}
	Suppose $S$ has maximum search density, $f$ has total curvature $\kappa_f$ and $g^\#$ has total curvature $\kappa_{g}$. Then for all $A \subset S$, the density $\rho(A)$ satisfies
	\[
	(1-\kappa_f)(1-\kappa_{g^\#}) \rho^* \le \rho(A) \le \rho^*.
	\]
\end{lemma}
\begin{proof}
The second inequality follows from the fact that $S$ has maximum density of $\rho^*$. To prove the first inequality, first observe that
\[
\frac{f(S)-f(\bar{A})}{f(A)} \ge \frac{\sum_{s \in A} f_{S-s}(s)}{\sum_{s \in A} f(s)} \ge \frac{\sum_{s \in A}(1-\kappa_f)f(s)}{\sum_{s \in A} f(s)} = 1-\kappa_f.
\]
Similarly, $g^\#(S) - g^\#(\bar{A}) \ge (1-\kappa_{g^\#}) g^\#(A)$, or, equivalently, $g(A) \ge (1-\kappa_{g^\#})(g(S)-g(\bar{A}))$.
Hence
\[
(1-\kappa_f)(1-\kappa_{g^\#})\rho^* f(A) \le (1-\kappa_{g^\#})\rho^*(f(S)-f(\bar{A})) \le (1-\kappa_{g^\#})(g(S)-g(\bar{A})) \le g(A),
\]
where the second inequality comes from the fact that $S$ has maximum density. The first inequality of the lemma follows.
\end{proof}
We can now revisit the proof of Lemma~\ref{lem1} by deriving a tighter upper bound on the expected cost of any search strategy $\pi$ and a tighter lower bound on an optimal search strategy $\pi^*$ when $S$ has maximum density. 
This is based on Edmonds's well-known greedy algorithm (see e.g.~\cite[Section 3.2]{Fujishige}). The submodular base polyhedron is defined as:
\begin{equation}\label{eq:submod-polytope}
{\mathbb{B}(f) = \{ \mathbf x \in \mathbb{R}^S: \mathbf x(A) \le f(A) \text{ for all }A \subset S, \mathbf x(S)=f(S)\}}
\end{equation}
\begin{lemma}[Edmonds] For a submodular function $f$, an optimal solution to 
	$\max \mathbf w ^T \mathbf x$  subject to $\mathbf{x}\in \mathbb{B}(f)$ is given by
	\[
	x_j=f(S_j^\pi)-f(S_j^{\pi}-j),\quad \forall j=1,\ldots,n,
	\]
	where $\pi$ is a permutation that orders $S$ in non-increasing order of $w_j$.\label{lem:edmonds}
\end{lemma}
\begin{lemma} \label{lem:epsilon}
	Suppose $f:2^S \rightarrow \mathbb R^+$ is submodular and $g \rightarrow \mathbb R^+$ is supermodular, and let $\kappa_f$ and $\kappa_{g^\#}$ be the total curvature of $f$ and $g^{\#}$, respectively. Define a function $\varepsilon=\varepsilon_{f,g}$ on permutations $\pi$ of $S$ by
	\[
	\varepsilon(\pi) =  \sum_{j=1}^n (f(S_j^\pi) - f(S_j^\pi - j)) (g(S_j^\pi) - g(S_j^\pi - j)) = \sum_{j=1}^n d_jf(S_j^\pi) d_j g(S_j ^\pi).
	\]
	Let $\pi_1$ be a permutation that orders the elements in non-increasing order of $f(j)$ and let $\pi_2$ be a permutation of $S$ that orders the elements in non-increasing order of $g^{\#}(j)$. Then
	\begin{enumerate}[(i)]
		\item $ (1-\kappa_f) \varepsilon(\pi_1) \le \min_\pi \varepsilon(\pi)$ and
		\item $(1-\kappa_{g^\#}) \varepsilon(\pi_2) \le  \min_\pi \varepsilon(\pi)$.
	\end{enumerate}
\end{lemma}
\begin{proof}
For part (i), let us fix the cost function $w_j=f(j)$ for $j=1,\ldots,n$.
Then, Lemma~\ref{lem:edmonds} implies that $\pi_1$ minimizes the function\[
\varepsilon'(\pi) = \sum_{j=1}^n w_j (g(S_j^\pi) - g(S_j^\pi - j)).
\]
It follows that for any permutation $\pi$,
\[
(1-\kappa_f) \varepsilon(\pi_1) \le (1-\kappa_f) \varepsilon'(\pi_1) \le (1-\kappa_f)\varepsilon'(\pi) \le  \varepsilon(\pi),
\]
third inequality follows from the definition of $\kappa_f$.

Part (ii) follows using the similar argument,
or by observing that $\varepsilon_{f,g}(\pi) = \varepsilon_{g^{\#},f^{\#}}(\pi')$, where $\pi'$ is the reverse permutation of $\pi$ (so that $\pi'(i) = \pi(n+1-i)$). Indeed,
\begin{align*} 
\varepsilon_{f,g}(\pi) &= \sum_{j=1}^n (f(S_j^\pi) - f(S_j^\pi - j)) (g(S_j^\pi) - g(S_j^\pi - j)) \\
&= \sum_{j=1}^n (f^\#(\overline{S_j^\pi -j} ) - f^\#(\overline{S_j^\pi })) (g^\#(\overline{S_j^\pi -j}) - g^\#(\overline{S_j^\pi} )) \\
& = \sum_{j=1}^n (f^\#(S_j^{\pi'}) - f^\#(S_j^{\pi'} - j)) (g^\#(S_j^{\pi'}) - g^\#(S_j^{\pi'} - j)) \\
& = \varepsilon_{f^\#,g^\#}(\pi').
\end{align*}
\end{proof}

\begin{theorem} \label{thm:curvature}
	Suppose that the submodular function $f$ and the supermodular function $g$ are given by a value oracle, $f$ has total curvature $\kappa_f <1$, and $g^\#$ has total curvature $\kappa_{g^\#}<1$.
	Then there is a search strategy that can be computed in strongly polynomial time and approximates an optimal search strategy for the submodular search problem
	with approximation ratio $ \frac{ 2}{1+ \delta}$, where 
	\[
	\delta = \min \left \{ \theta, \frac{2 \theta \max \{1- \kappa_f, 1- \kappa_{g^\#}\}}{1+\theta} \right \},
	\]
	and $\theta = (1-\kappa_f)(1-\kappa_{g^\#})$. If either $f$ or $g$ is modular then the approximation ratio is $\frac{2}{1+\theta}$.
\end{theorem}

\begin{proof}
First suppose that $S$ has maximum density. We normalize $f$ and $g$ so that $f(S)=g(S)=1$, thus $\rho^*=\rho(S)=1$.

Recall that by duality, 
$\min_\pi c_{f,g}(\pi) = \min_\pi c_{g^{\#},f^{\#}} (\pi)$. Note that $S$ has maximum density with respect to $f$ and $g$ if and only if it has maximum density with respect to $g^{\#}$ and $f^{\#}$. 

Hence, by Lemma~\ref{lem:curv}, for any $A \subset S$,
\[
\theta \le \frac{g(A)}{f(A)} \le 1 \text{ and } \theta \le \frac{f^{\#}(\bar{A})}{g^{\#}(\bar{A})} \le 1.
\]
This means, in particular, that 
\begin{align}
f(A) \le \min \left \{ \frac{g(A)}{\theta}, 1-\theta+ \theta g(A) \right \}. \label{eq:upper-env}
\end{align}
For any search $\pi$, we can write
\begin{align}
c(\pi) &=   \sum_{j=1}^n (g(S_j) - g(S_j - j)) f(S_j) \nonumber\\
&= \frac 1 2 \varepsilon(\pi) + \sum_{j=1}^n \frac 1 2 (g(S_j) - g(S_{j-1})) (f(S_j) + f(S_{j-1})). \label{eq:upper-env2}
\end{align}
The sum in (\ref{eq:upper-env2}) is the area under the piecewise linear curve in $\mathbb R^2$ connecting the points $(g(S_j),f(S_j)), j=0,1,\ldots,n$. By (\ref{eq:upper-env}), this is at most the area under the curve ${y = \min \left \{ x/\theta, 1-\theta+ \theta x \right \}}, x \in [0,1]$, which can be easily calculated to be $1/(1+ \theta)$. 

Since the expected cost is always bounded above by $1$, it follows that
\[
c(\pi) \le \min \left  \{\frac 1 {1+\theta} + \frac 1 2 \varepsilon(\pi), 1 \right \}.
\]
Now consider an optimal search $\pi^*$. For this search, the sum in~(\ref{eq:upper-env2}) is at least $1/2$, since ${f(A) \ge g(A)}$ for any $A \subset S$. By Lemma~\ref{lem:epsilon} we can choose $\pi$ to be some search such that ${\varepsilon(\pi^*)  \ge \max \{1- \kappa_f, 1- \kappa_{g^\#}\}\varepsilon(\pi)}$. So (\ref{eq:upper-env2}) implies that
\[
c(\pi^*) \ge \frac 1 2 + \frac 1 2 \max \{1- \kappa_f, 1- \kappa_{g^\#}\} \varepsilon(\pi).
\] 
Hence
\[
\frac{c(\pi)}{c(\pi^*)} \le \frac {\min \left  \{\frac 1 {1+\theta} + \frac 1 2 \varepsilon(\pi), 1 \right \}}{\frac 1 2 + \frac 1 2 \max \{1- \kappa_f, 1- \kappa_{g^\#}\} \varepsilon(\pi)}.
\]
This is maximized either at $\varepsilon(\pi) = \frac{2\theta}{1+\theta}$ or $\varepsilon = 0$, giving the first bound in the statement of the theorem.

If either $f$ or $g$ is modular then $	\delta = \min \left \{ \theta, \frac{2 \theta }{1+\theta } \right \} = \theta$.

If $S$ does not have maximum density, then a similar induction argument to that of Theorem~\ref{thm2} completes the proof.
\end{proof}

We note that we would be able to improve the approximation ratio in Theorem~\ref{thm:curvature} to $\frac{2}{1+\theta}$ for arbitrary submodular $f$ and supermodular $g$ if we could find an exact solution to the problem of minimizing $\varepsilon(\pi)$ of Lemma~\ref{lem:epsilon}, and we leave this as an open problem. 

\subsection{An optimal search for series-parallel decomposable problems} \label{sec:reducible}

In this section we show how Theorem~\ref{thm1} may be used to determine an optimal search for problems we call {\em series-parallel decomposable}. 
The idea for series-parallel decomposability is motivated by the following example of expanding search on a tree, considered in \cite{AlpernLidbetter}. Let $S$ be the vertex set of a tree $T=(S,E)$ with edge set $E$ and each $e\in E$ has weight $w(e)$. Let $r\in S$ be the root of the tree and restrict attention to searches that begin at $r$. For a set of edges $A$, define $f(A)$ to be the sum of the edge weights in the
tree that is spanned by $\{r\}\cup A$. It is clear that if $r$ has degree $1$ then every search begins with the edge incident to $r$. We generalize this principle by defining {\em $f$-initial sets} below. If $r$ has degree greater than $1$, then $T$ is the union of two edge-disjoint subtrees with root $r$, and it is easy to show that there is a maximal density subset of $S$ whose elements are the vertices of one of these subtrees. So the problem of finding an optimal search can be decomposed. We generalize this principle using the concept of separators. We say that a proper non-empty subset $B\subset S$ is a {\em separator} of $f$ if $f$ is the direct sum of $f|_B$ and $f|_{\overline{B}}$. In order to check that $B$ is an $f$-separator, we only need to verify that $f(S)=f(B)+f(\bar{B})$ (see \cite[Proposition 5]{Cunningham}). 

\paragraph{The $f$-initial sets}
For a set $A\subset S$, we define the {\em $f$-closure} $cl(A)$ of $A$ as the maximal set $B$ containing $A$ such that $f(B)=f(A)$; there is a unique such set. We say that a proper subset $I \subset S$ is an {\em $f$-initial set} if $I \subset cl({s})$ for every $s \in \bar{I}$. In the case that $f$ corresponds to the special case of precedence-constrained scheduling, $f$-initial sets and $f$-closures correspond to the usual notions of an initial sets and closures with respect to the precedence constraints. We leave it to the reader to check that a set $I$ is an $f$-initial set if and only if for any subset $A \subset S$ that contains some element of $\bar{I}$, we have $f(A \cup I) = f(A)$.

Note that if there is an $f$-initial set, then the total curvature of $f$ is $1$ (that is, the worst possible). Therefore our approximation given in Theorem~\ref{thm:curvature} is not helpful. However, it is easy to show that there is an optimal search with initial segment $I$.
\begin{lemma}\label{lem:initial}
	If $I$ is an $f$-initial set or $\bar{I}$ is a $g^\#$-initial set, then there exists an optimal search with initial segment $I$.
\end{lemma}
\begin{proof}
First suppose that $I$ is an $f$-initial set, and that there are no optimal searches with initial segment $I$. Let $\sigma$ be an optimal search that has been chosen to minimize $\sum_{s \in I} \sigma^{-1}(s)$. Since $I$ is not an initial segment of $\sigma$, there must be some $t \notin I $ that directly preceeds some $s \in I$. Let $A$ be the set of all elements preceeding $t$ and let $\tau$ be the search obtained by switching the order of $s$ and $t$.

Then the difference in expected costs between $\sigma$ and $\tau$ is
\begin{align*}
c(\sigma) - c(\tau) & = f(A \cup \{t\}) d_t g(A \cup \{t\})  + f(A \cup \{s,t\} )d_s g(A \cup \{s,t\}) \\
& \quad - f(A \cup \{s\}) d_s g(A \cup \{s\}) - f(A \cup \{s,t\} )d_t g(A \cup \{s,t\}) \\
&\ge  f(A \cup \{s,t\} )(d_t g(A \cup \{t\}) + d_s g(A \cup \{s,t\}) - d_t g(A \cup \{s,t\}) ) - f(A \cup \{s\}) d_s g(A \cup \{s\}) \\
& = d_t f (A \cup \{s,t\}) d_s g(A \cup \{s\})\\
& \ge 0,
\end{align*}
where the first inequality comes from the fact that $f(A \cup \{t\}) = f(A \cup \{t\} \cup I) \ge f(A \cup \{s,t\})$, since $I$ is an initial set and by monotonicity. Hence, $\tau$ is an optimal search with $\sum_{s \in I} \tau^{-1}(s) < \sum_{s \in I} \sigma ^{-1}(s)$, contradicting the definition of $\sigma$. So there must be an optimal search with initial segment $I$.

If $\bar{I}$ is a $g^\#$-initial set, the fact that there is an optimal search beginning with initial segment $I$ follows immediately from duality. 
\end{proof}

Therefore, if an $f$-initial set $I$ exists, then it follows from Lemma~\ref{lem:decomp} that in order to find an optimal search of $S$ it is sufficient to find an optimal search of $I$ with respect to $f|_I$ and $g|_I$ and an optimal search of $\bar{I}$ with respect to $f_{\overline{I}}$ and $g_{\overline{I}}$. Similarly if $\bar{I}$ is $g^\#$-initial. This is one way in which the problem can be decomposed.

Finding an $f$-initial set can be performed in polynomial time, as we now explain. For any $s \in S$, let $I_s$ be the largest $f$-initial set not containing $s$ (if no such $f$-initial set exists, let $I_s = \emptyset$). If we find a nonempty $I_s$ for any $s\in S$, we can return it as an $f$-initial set. In case $I_s=\emptyset$ for every $s\in S$, we conclude that  there is no $f$-initial set.

In order to find $I_s$, we maintain a candidate $T$, starting with $T=cl(s) - \{s\}$. We know that $I_s \subset T$, and that for every $t \in S -T$, we have $I_s \subset cl(t)$. Hence we take an arbitrary $t \in S - T$ and update $T$ as $T \cap cl(t)$.

We iterate this process: while there exists a $t \in S - T$ that we have not yet examined, we update $T$ as $T \cap cl(t)$. We examine every $t$ at most once. At termination, if $T \neq \emptyset$ then we must have $T \subset cl(t)$ for every $t \in S - T$, showing that $T$ is an $f$-initial set. It is also clear from the construction that $T=I_s$, the largest $f$-initial set disjoint from $s$.

\paragraph{Separators}
The other way that the problem can be decomposed is by finding a separator, as we now explain. For a lattice $\mathcal L$ and $B \subset S$
the restriction to $B$ is $\mathcal L|_B=\{A\in\mathcal L\colon A\subset B\}$.
If $S$ and $T$ are disjoint subsets and if $\mathcal L$ is a lattice in $S$ and $\mathcal N$ is a lattice in $T$, then the
direct sum of these lattices is
$\left\{A\cup A'\colon A\in\mathcal L,\ A' \in \mathcal N\right\}$.
It is a lattice in $S\cup T$.

\begin{lemma}\label{lem:sep}
	If $B$ is a separator of both $f$ and $g$, then $\mathcal M$ is the direct sum of $\mathcal M|_B$ and $\mathcal M|_{\overline{B}}$.
\end{lemma}
\begin{proof}
If $A$ has maximum density, then the inequality
\[\rho^*=\rho(A)=\frac{g(A\cap B)+g(A\cap \bar{B})}{f(A\cap B)+f(A\cap \bar{B})}\leq \max\{\rho(A\cap B),\rho(A\cap \bar{B})\}\]
is in fact an equality. It follows that if $A\cap B\neq\emptyset$ then $\rho(A\cap B)=\rho^*$, and if $A\cap \bar{B}\neq\emptyset$, then $\rho(A\cap \bar{B})=\rho^*$. 
\end{proof}

It follows from Lemma~\ref{lem:sep} that if $B$ is a separator then either $B$ or $\bar{B}$ (or both) must contain a subset $A$ of density $\rho^*$. So by Theorem~\ref{thm1}, there exists an optimal search of $S$ with initial segment $A$, where $A$ is a proper subset of $S$. In that case, we can again apply Lemma~\ref{lem:decomp} to decompose the problem of finding an optimal search into two subproblems on $A$ and $\bar{A}$.

If there exists a separator of both $f$ and $g$, then it is possible to find one in strongly polynomial time, using the following method. The {\em connectivity function} of $f$ is defined as $d_f(B)=f(B)+f(\bar{B})-f(S)$.
It is a symmetric non-negative submodular function~\cite{Cunningham}, as is the connectivity function of $h=f-g$.
We say that a non-empty subset $B\subset S$ is a {\em split} if $d_h(B)$ is minimal and $d_h(A)>d_h(B)$ for
all non-empty $A\subset B$. Obviously,
a split is a separator of both $f$ and $g$ if and only if  $d_h(B)=0$, and since a split can be computed from a submodular
function minimization, this can be carried out in strongly polynomial time by Queyranne's algorithm for minimizing symmetric, submodular functions~\cite{Queyranne}. 


We can now define series-parallel decomposability, which is an extension of an idea from Theorem~3 of~\cite{FKR}.

\begin{definition}
	\label{def:reducible}
	We say the submodular search problem is {\em series-parallel decomposable} if 
	\begin{enumerate}[(i)]
		\item  (series decomposable) there exists some set $I$ such that $I$ is $f$-initial or $\bar{I}$ is $g^\#$-initial, or
		\item (parallel decomposable) there exists some set $B$ that is a separator of both $f$ and $g$.
	\end{enumerate}
	We say that $f$ is {\em series-parallel decomposable} if it can be repeatedly decomposed until all remaining search segments are singletons. 
\end{definition} 

We have shown that if the submodular search problem is series-parallel decomposible, then by decomposing it we can determine an optimal search strategy. We summarize this result with a theorem.

\begin{theorem}\label{thm3}
	Suppose that a submodular function $f$ and the supermodular function $g$ are given by a value oracle. Then we can decide in strongly polynomial time whether the submodular search problem is series-parallel decomposable, and if so we can find an optimal search.
\end{theorem}

As pointed out at the end of Subsection~\ref{sec:2approx}, the submodular search problem can be solved if both $f$ and $g$ are modular. In this case the problem is series-parallel decomposable (by repeated parallel decompositions). The example from \cite{AlpernLidbetter} described at the beginning of this subsection is series-parallel decomposable as well. Theorem~\ref{thm3} also has an interpretation in scheduling, with regards to scheduling jobs whose precedence constraints are given by a series-parallel graph. This is explained in more detail in Subsection~\ref{sec:scheduling}.

We could extend the concept of series-parallel decomposition to a more general notion of {\em logarithmic decomposition}, meaning that the problem can be repeatedly decomposed until all remaining search segments $A$ have cardinality $|A| \le p \log n$, for some (small) constant $p$. Then each search subproblem in such a subset $A$ can be solved by brute force (in $|A|!$ time) or by dynamic programming (in $|A|2^{|A|}$ time; see~\cite{HeldKarp}), resulting in overall time of $O(n^{p+2})$.

\subsection{Applications to Scheduling}
\label{sec:scheduling}
As we outlined in Subsection~\ref{sec:ex}, the submodular search problem has a natural application to single machine scheduling. Theorem~\ref{thm3} generalizes the well-known result that the problem $1\mid prec\mid \sum w_j C_j$ can be solved in polynomial time if the Hasse diagram defined by the precedence constraints on the jobs is a generalized series-parallel graph~\cite{Adolphson,Lawler78}. We define {\em generalized series-parallel} here for completeness. 

Denote a Hasse diagram by a pair $\{N,E\}$ of nodes $N$ and directed edges $E \subset N^2$. For disjoint vertex sets $N_1$ and $N_2$, we let $N_1\times N_2=\{(u,v): u\in N_1, v\in N_2\}$ denote the complete directed bipartite graph from $N_1$ to $N_2$. Then
\begin{enumerate}
	\item If $G_1 = \{N_1, E_1\}$ and $G_2 = \{N_2, E_2\}$ are graphs on disjoint vertex sets, then $\{N_1 \cup N_2, E_1 \cup E_2\}$ is the {\em parallel composition} of $G_1$ and $G_2$.
	\item If $G_1 = \{N_1, E_1\}$ and $G_2 = \{N_2, E_2\}$ are graphs on disjoint vertex sets, then $\{N_1 \cup N_2, E_1 \cup E_2 \cup (N_1 \times N_2)\}$ is the {\em series composition} of $G_1$ and $G_2$.
\end{enumerate}

The graph $\{\{i\},\emptyset\}$ containing a single node is generalized series-parallel, and any graph that can be obtained by a finite number of applications of parallel composition or series composition of generalized series-parallel graphs is generalized series-parallel.

Recall that in the framework of the submodular search problem, the problem $1|prec|\sum w_j C_j$ has cost function $f$ such that $f(A)$ is the sum of the processing times of the jobs in the precedence closure of a set $A$ of jobs. If the Hasse diagram corresponding to some precedence constraints is a parallel composition then clearly the corresponding submodular cost function $f$ has a separator. If the Hasse diagram is a series composition then $f$ has an $f$-initial set. Thus the concept of a series-parallel decomposable submodular search problem generalizes the problem $1|prec|\sum w_j C_j$ when the precedence constraints are given by a generalized series-parallel graph. 

It is also possible that the concept of series-parallel decomposability could be used to generalize work in the machine scheduling literature that extends the solution of $1|prec|\sum w_jC_j$ in the generalized series-parallel case, for example~\cite{Sidney-Steiner,Momma-Sidney}. However, this work does not correspond directly with Theorem~\ref{thm3}, and we leave for future work the question of how to link these extensions to our submodular framework.

Theorem~\ref{thm2} can also be applied to the problem $1\mid prec\mid \sum w_A h(C_A)$ in which the object is to minimize the weighted sum of some concave function $h$ of the completion times of {\em subsets} of jobs, where $h$ is given by a value oracle. This problem also has a natural interpretation in terms of searching for multiple hidden objects, where $w_A$ is the probability there are objects hidden in the locations contained in $A$. We summarize this in the following theorem.

\begin{theorem}\label{thm4}	Suppose the non-decreasing, concave real function $h$ and the function ${g: A \mapsto \sum_{B \subset A} w_B}$ are given by a value oracle, where the weights $w_A$ are non-negative. Then there is an algorithm running in strongly polynomial time in $n$ that computes a {$2$-approximation} for $1\mid prec\mid \sum w_A h(C_A)$.
\end{theorem}
\begin{proof}
This is an immediate consequence of Theorem~\ref{thm2}, due to the fact that the composition of the concave function $h$ with the submodular function $f:A \mapsto p(\widetilde A)$ is itself submodular and the function $g: A \mapsto \sum_{B \subset A} w_B$ is supermodular. 
\end{proof}

It is worth noting since $h$ is applied to costs in this model, a concave $h$ corresponds to larger losses having a decreasing marginal disutility. This reflects a risk seeking attitude, a less common assumption than risk aversion.

We may also consider the more specific problem $1||\sum w_j h(C_j)$ for some non-decreasing function $h$. In~\cite{Megow}, a polynomial time approximation scheme is given for the problem. If $h(C_j)=C_j^\beta$ for $\beta \neq 1$, then it is
the problem considered in \cite{Bansal}. It is unknown whether there exists a polynomial
time algorithm to compute an optimal schedule for this problem, or if it is NP hard (see \cite{Bansal}
and the references therein). For concave or convex $h$, it is shown in~\cite{Stiller} that Smith's rule (for the original problem $1||\sum_j w_j C_j$) yields a $(\sqrt{3}+1)/2 \approx 1.37$-approximation to the problem $1|| \sum_j w_j h(C_j)$, while~\cite{Hohn} gives explicit formulas for the exact approximation ratio of Smith's rule for this problem in terms of a maximization over two continuous variables. 

Of course, our algorithm gives rise to a Sidney decomposition for $1|| \sum_j w_j h(C_j)$ that is not necessarily consistent with the application of Smith's rule for the problem $1||\sum_j w_j C_j$. Furthermore, Theorem~\ref{thm1} implies that every optimal schedule must follow a Sidney decomposition of our type. If $h$ defines a submodular cost function $f$ with low total curvature, we can use Theorem~\ref{thm:curvature} to express the approximation ratio of our algorithm in a simple form that may be better than the $1.37$-approximation in~\cite{Stiller}. For example, if $\kappa_f = 1/2$ and $\kappa_{g^\#}=0$ then our approximation ratio is $2/(1+1/2) = 1.33$. Note that the curvature $\kappa_f$ of $f$ is given by
\[
1-\kappa_f = \min_A \frac{h(p(S))- h(p(A))}{h(p(S)-p(A))}  \ge \inf_{y \in [0,p(S)]} \frac{h(p(S)) - h(y)}{h(p(S)-y)},
\]
where $p:2^S \rightarrow \mathbb R$ denotes processing time. Since the fraction on the right is the ratio of a decreasing concave function to a decreasing convex function, its infinum is achieved in the limit as $y \rightarrow p(S)$. 
Since $g^{\#}$ is modular in this problem, applying Theorem~\ref{thm:curvature}, we obtain the following.
\begin{theorem}\label{thm:noprec}
	Suppose $h:\mathbb{R}^+ \rightarrow \mathbb{R}^+$ is a non-decreasing concave real function. Then there is a schedule for $1||\sum w_j h(C_j)$ that can be computed in strongly polynomial time and approximates an optimal schedule with approximation ratio $\frac{2}{2-\kappa_f}$. The parameter $\kappa_f$ is given by
	\[
	1-\kappa_f = \lim_{y \rightarrow p(S)}  \frac {h(p(S))-h(y)}{h(p(S)-y)} = \frac {h'(p(S))}{h'(0)},
	\]
	where the second equality holds by l'H\^{o}pital's rule if $h$ is differentiable at $0$ and $p(S)$.
\end{theorem}

By way of an example, we scale so that $p(S)=1$ and take the standard log utility function \cite{Neumann}: $h(y)=\log(1+ay)$, for some positive constant $a$. Then Theorem~\ref{thm:noprec} implies that we can find a schedule that approximates an optimal schedule for $1||\sum w_j h(C_j)$ with approximation ratio $1+a/(2+a)$ .

Another classic example is the function $h(y) = (1-e^{-ry})/r$ (with discount rate $r >0$) for continuously discounted search (or wait) time, as in~\cite{Rothkopf}. For this choice of $h$, our approximation ratio for $1||\sum w_j h(C_j)$ is  $2/(1+e^{-r})$.

We conclude this section by arguing that the submodular search problem really is more general than $1\mid prec\mid \sum w_A h(C_A)$. Indeed, consider the problem with $S=\{1,2,3\}$ and $f(1)=f(2)=f(3)=1$, $f(1,2)=f(1,3)=2$, $f(2,3)=3/2$, $f(1,2,3)=2$. Suppose $f$ is defined by some partial order on $S$, some processing times $p_j$ and some concave function $h$ of completion times, as in the proof of Theorem~\ref{thm4}. Then the partial order on the jobs must be an antichain (that is, no jobs are comparable), since otherwise we would have $f(A)=f(B)$ for some $1=|A| \subsetneq |B|$. It must also be the case that $p_1=p_2=p_3$, because if not, by the concavity of $h$ and because $f(1)=f(2)=f(3)=1$, it would have to be the case that $f(A)=1$ for all $A \neq \emptyset$. But then $2=f(1,2)=h(p_1+p_2)=h(p_2+p_3)=f(2,3)=3/2$, a contradiction.

\section{The Submodular Search Game}
\label{sec:game}

We now turn to the submodular search game and we seek optimal mixed strategies for the players, settling a question from~\cite{FKR}. Here, each Hider's mixed strategy (probability distribution of $S$) $\mathbf x$ defines a modular function $g$ where $g(A) = \mathbf x(A)$ for all $A \subset S$.
We use $c_{f,\mathbf{x}}$ to denote the search cost for such a $g$.
A mixed strategy for the Searcher is some $\mathbf p$ which assigns a probability $\mathbf p(\pi)$ to each pure strategy $\pi$. We denote by $C_f(\mathbf p,\mathbf{x})$ the expected search cost for mixed strategies $\mathbf p$ and $\mathbf{x}$, where $\mathbf{p}$ and $\mathbf{x}$ are independent. 
That is,
\[
C_f(\mathbf{p},\mathbf{x})=\sum_{\pi} \mathbf p(\pi) c_{f,\mathbf{x}}(\pi).
\]
We suppress the subscript $f$ when the context is clear.

Recall the definition of the  {\em submodular base polyhedron} $\mathbb{B}(f)$ in
\eqref{eq:submod-polytope}.
We apply Theorem~\ref{thm1} to settle a question from \cite{FKR}.

\begin{theorem}
	\label{cor:hider}
	Every equilibrium strategy for the Hider in the submodular search game is in the scaled base polyhedron $\frac 1 {f(S)} \mathbb{B}(f)$.
\end{theorem}
\begin{proof}
By contradiction.
Suppose that $\mathbf x$ is an equilibrium Hider strategy, but it is not in the base polyhedron. Then
there exists a subset $A$ such that $\mathbf x(A)>f(A)/f(S)$, so that $\rho(A)>1/f(S)=\rho(S)$. It follows that the largest subset $M$ of maximum density is a proper subset of $S$.
Any pure strategy best response $\pi$ to $\mathbf x$ searches $M$ first, by Theorem~\ref{thm1}, and hence an optimal mixed strategy of the Searcher assigns positive probability only to orderings of $S$ with initial segment $M$. Now informally, an optimal response
to these Searcher strategies is to hide in $\bar{M}$, which cannot be an equilibrium strategy. 

More formally, we observe that every pure search strategy in the support of an equilibrium strategy $\mathbf p$ is a best response to $\mathbf x$, so it must start by searching the whole of $M$. Let $k \in \bar{M}$, then we must have $f(M\cup\{k\})>f(M)$, otherwise $M$ cannot be maximal. Define $\mathbf y$ by $\mathbf y(M)=0,\mathbf y(k)=\mathbf x(k)+\mathbf x(M)$ and $\mathbf y(j)=\mathbf x(j)$ for any $j \notin M\cup\{k\}$. Then we have
\[
C(\mathbf p,\mathbf y)-C(\mathbf p,\mathbf x) \ge \mathbf x(M)(f(M\cup\{k\})-f(M)) >0,
\]
so $\mathbf x$ cannot be a best response to $\mathbf p$: a contradiction. Hence we must have $M=S$ and $\mathbf x$ is in the base polyhedron of $\frac{1}{f(S)}f$. 
\end{proof}

Combining Theorem~\ref{cor:hider} with Lemma~\ref{lem1}, the value of the submodular search game must lie between $f(S)/2$ and $f(S)$. Also, any Hider strategy $\mathbf x$ in the base polyhedron of $\frac{1}{f(S)} f$ is a {\em $2$-approximation for the Hider's equilibrium strategy}, in the sense that $\min_{\mathbf p} C(\mathbf p, \mathbf x) \ge V/2$, where $V$ is the value of the game. Furthermore, {\em any} Searcher strategy $\mathbf p$ is a {\em $2$-approximation for the Searcher's equilibrium strategy}, in the sense that $\max_{\mathbf x} C(\mathbf p, \mathbf x) \le 2V$.


We can also find strategies that are better approximations for the equilibrium strategies for cost functions with total curvature less than $1/2$. Define the modular function $h(A)=\sum_{s\in A}f(s)$. Then $f(A)\leq h(A)$, and $f(A)\geq (1-\kappa)h(A)$ for all $A \subset S$. It follows that for any mixed strategies $\mathbf p$ and $\mathbf{x}$, we have $C_h(\mathbf p,\mathbf{x}) \ge C_f(\mathbf p,\mathbf{x}) \ge (1-\kappa) C_h(\mathbf p,\mathbf{x})$.

Two different solutions to the search game with modular cost function $h$ can be found in~\cite{AlpernLidbetter} and~\cite{Lidbetter}. The equilibrium strategy for the Hider (shown to be unique in~\cite{AlpernLidbetter}) is given by $\mathbf{x}^h(s)=h(s)/h(S)=f(s)/h(S)$. The equilibrium strategy $\mathbf p^h$ for the Searcher given in~\cite{Lidbetter} is to begin with an element $s$ with probability $\mathbf{x}^h(s)$ and to search the remaining elements in a uniformly random order. Note that this not an extreme point solution to the LP defining an equilibrium Searcher strategy. The equilibrium strategy for the Searcher given in~\cite{AlpernLidbetter} is less concise to describe, and is iteratively constructed, much like the Searcher strategy of Theorem~\ref{thm5}, which generalizes it.

\begin{proposition} \label{prop:game}
	Suppose $f$ has total curvature $\kappa < 1/2$. Then the equilibrium strategies $\mathbf{x}^h$ and $\mathbf p^h$ in the submodular search game with cost function $h$ are $\left( \frac 1{1-\kappa} \right)$-approximations for the equilibrium strategies in the submodular search game with cost function $f$.
\end{proposition} 

\begin{proof}
Let $V_f$ and $V_h$ be the value of the game with cost function $f$ and the game with cost function $h$, respectively. Then
\[
V_h = \max_{\mathbf{x}} C_h(\mathbf p^h,\mathbf{x}) \ge \max_{\mathbf{x}} C_f(\mathbf p^h,\mathbf{x}) \ge V_f \ge \min_{\mathbf p} C_f(\mathbf p,\mathbf{x}^h) \ge (1-\kappa) \min_{\mathbf p} C_h(\mathbf p,\mathbf{x}^h)
=(1-\kappa)V_h.
\]
It follows that $\max_{\mathbf{x}} C_f(\mathbf{p}^h,\mathbf{x})\le V_h \le V_f/(1-\kappa)$ and $\min_{\mathbf{p}} C_f(\mathbf p,\mathbf{x}^h) \ge (1-\kappa)V_h \ge (1-\kappa)V_f$. 
\end{proof}

We now define a notion of series-parallel decomposition for the submodular search game, similar to Definition~\ref{def:reducible} for the submodular search problem. We say the game is series-parallel decomposable if there is an $f$-initial set or if there is a separator of $f$. Note that this is equivalent to saying that the problem of finding a best response to a given Hider strategy $\mathbf x$ is series-parallel decomposable. 

In the case the game is series-parallel decomposable, we can improve upon Proposition~\ref{prop:game}.

\begin{theorem}\label{thm5}
	Suppose the submodular search game is series-parallel decomposable. Then if $f$
	is given by a value oracle, an equilibrium strategy $\mathbf{x}^f$ for the Hider
	can be computed in strongly polynomial time. An equilibrium Searcher strategy 
	can also be computed in strongly polynomial time. The value $V$
	of the game is
	\begin{align}
	V=\frac 1 2 (f(S)+\Phi), \label{eq:value}
	\end{align}
	where $\Phi = 
	\sum_{s \in S} \mathbf{x}^f(s) f(s)$.
\end{theorem}

\begin{proof}
The theorem is proved by induction on the number of hiding locations, $n=|S|$. We write $V=V^f$ and $\Phi = \Phi^f$ to indicate the dependence on $f$. We will define both the Hider's equilibrium strategy $\mathbf x^f$ and an equilibrium strategy $\mathbf p^f$ for the Searcher recursively.

The base case, $n=1$, is immediate, since for $S= \{s\}$ the players both have only one available strategy: $\mathbf x^f (s) = 1$ for the Hider and $\mathbf p^f(\pi)$ for the Searcher, where $\pi$ is the unique permutation of $S$. Then $\Phi^f = f(s) = f(S)$ and $f(S) = V^f = \frac 12 (f(S) + \Phi^f)$.  

For the induction step, there are two cases. The first case is that there is an $f$-initial set $I$. In this case, we claim that $V^f = f(I) + V^{f_I}$. Indeed, the Searcher can ensure that $V^f \le f(I) + V^{f_I}$ by using the strategy $\mathbf p^f$ which searches $I$ in any order then searches $\bar{I}$ according to the mixed strategy $\mathbf p^{f|_I}$. By Lemma~\ref{lem:initial}, the Hider can ensure that $V^f \ge f(I) + V^{f_I}$ by using the strategy $\mathbf x^f$ given by $\mathbf x^f(s) = \mathbf x^{f_I}(s)$ for $s \in \bar{I}$ and $\mathbf{x}^f(s)=0$ for $s \in I$. Hence, by induction, the strategies $\mathbf x^f$ and $\mathbf p^f$ are equilibrium strategies and can be calculated in strongly polynomial time. Furthermore,
\begin{align*}
V^f &= f(I) + \frac 1 2 \left(f_I(\bar{I})+\Phi_{f_I}(\bar{I})\right)\\
& = f(I) + \frac 1 2 \left( (f(S)-f(I)) + (\Phi_f(S) - f(I)) \right) \\
& = \frac 1 2 (f(S)+\Phi).
\end{align*}

The second case is that $f$ has a separator $A$. Then we define $\mathbf x^f$ on $A$ by $\mathbf x^f(s) = \frac{f(A)}{f(S)} \mathbf x^{f|_{A}}(s)$, and on $\bar{A}$ by $\mathbf x^f(s) = \frac{f(\overline{A})}{f(S)} \mathbf x^{f|_{\overline{A}}}(s)$. By Lemma~\ref{lem:sep}, there must be a set of maximum density contained in $A$, and by induction and Theorem~\ref{cor:hider}, $A$ has maximum density. So, by Theorem~\ref{thm1}, there is a best response $\pi$ to $\mathbf{x}^f$ that starts with $A$. By induction, we must have
\begin{align*}
C(\pi, \mathbf{x}^f) &\ge \mathbf{x}^f(A) V_{f|_A}(A) + \mathbf{x}^f(\bar{A})(f(A) + V_{f|_{\overline{A}}}(\bar{A})) \\
&=\frac{f(A)}{f(S)} \cdot \frac{1}{2} (f|_A(A)+\Phi_{f|_A}(A)) + \frac{f(\bar{A})}{f(S)} \cdot \left( f(A) + \frac 1 2 (f|_{\overline{A}}(\bar{A}) + \Phi_{f|_{\overline{A}}}(\bar{A})) \right) \\
& = \frac 1 2 \left( \frac{(f(A)+f(\bar{A}))^2}{f(S)}   +   \mathbf{x}^f(A) \Phi_{f|_A}(A) + \mathbf{x}^f(\bar{A}) \Phi_{f|_{\overline{A}}}(\bar{A})  \right) \\
& = \frac 1 2 (f(S) + \Phi),
\end{align*}
where the final equality comes from the fact that $\Phi_{f|_A}(A) =  \sum_{s \in A} \mathbf{x}^{f|_A}(s)f|_A(s) = \sum_{s \in A} \frac{\mathbf{x}^f(s)}{\mathbf{x}^f(A)} f(s)$, and similarly for $\Phi_{f|_{\overline{A}}}(A)$.

Now we turn to the Searcher's strategy $\mathbf{p}^f$, which, with probability $q$ searches $A$ first according to $\mathbf p^{f_A}$ and otherwise searches $\bar{A}$ first according to $\mathbf p^{f_{\overline{A}}}$, where
\[
q=\frac 1 2 + \frac  {\Phi_{f|_A}(A)-\Phi_{f|_{\overline{A}}}(\bar{A})} {2f(S)}.
\]
We prove by induction that this strategy ensures an expected search cost of at most $V$, where $V$ is given by Equation~(\ref{eq:value}). Let $s \in A$. Then, by induction, the expected search cost $C(\mathbf{p}^f,s)$ satisfies
\begin{align*}
C(\mathbf{p}^f,s) &\le V_{f|_A}(A) + (1-q)f(B)\\
& = \frac 1 2 \left(f(A)+\Phi_{f|_A}(A) \right) + \left( \frac 1 2 + \frac  {\Phi_{f|_B}(B)-\Phi_{f|_A}(A)} {2f(S)} \right) f(B) \\
& = \frac 1 2 \left( f(A) + f(B) + \mathbf{x}^f(A) \Phi_{f|_A}(A) + \mathbf{x}^f(\bar{A}) \Phi_{f|_{\overline{A}}}(\bar{A})  \right) \\
& = \frac 1 2 (f(S)+\Phi) = V. 
\end{align*}
This shows that the value is at most $V$. The case $s \in \bar{A}$ is similar, exchanging the roles of $A$ and $\bar{A}$. 
\end{proof}

Theorem~\ref{thm5} generalizes results in \cite{AlpernLidbetter}
on expanding search on a rooted tree, where it is shown that
the equilibrium Hider strategy is unique and can be computed efficiently, as can an equilibrium Searcher strategy. Expanding search on
a tree is a series-parallel decomposable submodular search game.

We may consider the submodular search game in the context of scheduling jobs with processing times and precedence constraints. One player chooses an ordering of jobs and the other player chooses a job; the payoff is the completion time of the chosen job. We can interpret this as a robust approach to scheduling, in which one job, unknown to the scheduler, has particular importance, and the scheduler seeks a randomized schedule that minimizes the expected completion time of that job in the worst case. This has a natural application to planning a research project or an innovation process, in which there are many directions the project can take, but it is unknown which task will be fruitful. Theorem~\ref{thm5} gives a solution of this scheduling game on series-parallel graphs. An interesting direction for future research would be to study the game on more general partial orders. 

\section{Final remarks}

We have shown that the notion of series-parallel decomposability is useful for solving both the submodular search problem and the submodular search game. A direction for future research could be to find some measure that captures the ``distance'' from being series-parallel decomposable, and show that better approximations can be found when the problem is close to being series-parallel decomposable. It is shown in~\cite{Ambuhl11} that better approximations to the single machine scheduling problem $1|prec|\sum w_j C_j$ can be found when the precedence constraints have low {\em fractional dimension}. It would be interesting to see whether this idea could be generalized to our setting.

The submodular search game that we have studied in this paper is a
zero-sum game between one Searcher and one Hider. In search games on networks,
one usually restricts attention to one Searcher only, since more
Searchers can divide up the space efficiently~\cite[p 15]{AlpernGal}.
However, in a submodular search game, such a division is impossible and it
is interesting to study games with multiple Searchers, which should relate to
multi machine scheduling problems.
Another extension would be to consider search games with selfish Hiders.
Selfish loading games have been studied, and an overview
can be found in~\cite{Vocking}. These are games between one Searcher
and multiple Hiders and a modular payoff function, similar to the
scheduling problem $1 || \sum  w_jC_j$. A study of
submodular search games with selfish Hiders would extend this to
$1\mid prec\mid \sum w_j h(C_j)$, for $h$ concave.

We end with a question. It is known that the complexity of determining an equilibrium
Searcher strategy in a specific search game on a network is NP hard~\cite{Stengel},
see also~\cite{Megiddo}. What is the complexity of determining equilibrium strategies in the submodular search game?

\section*{Acknowledgments.}
The authors would like to thank Christoph D\"urr for pointing out the connection between expanding search and scheduling. 

We would also like to thank three anonymous reviewers whose comments and suggestions inspired us to greatly improve the paper. In particular, we thank a reviewer for drawing our attention to the work of Pisaruk~\cite{Pisaruk92} and another reviewer for suggesting a more general notion of series-parallel decomposability and for corrections to the proof of Theorem~\ref{thm5}.

L\'aszl\'o A. V\'egh was supported by EPSRC First Grant EP/M02797X/1.


\begin{thebibliography}{10}
	
	\bibitem{Adolphson} Adolphson D (1977)  Single machine job sequencing with precedence constraints. {\em SIAM J. Comput.} 6:40--54.
	
	\bibitem{AlpernGal} Alpern S, Gal S (2003) The theory of search games and rendezvous. {\em Kluwer International Series in Operations Research and Management
		Sciences} (Kluwer, Boston).
	
	\bibitem{AlpernHoward} Alpern S, Howard JV (2000) Alternating search at two locations. {\em Dynamics and Control} 10(4):319--339.
	
	\bibitem{AlpernLidbetter} Alpern S, Lidbetter T (2013) Mining coal or finding terrorists: the expanding search paradigm. {\em Oper. Res.} 61(2):265--279.
	
	\bibitem{AlpernLidbetter2} Alpern S, Lidbetter T (2014) Searching a variable speed network. {\em Math. Oper. Res.} 39(3):697--711.
	
	\bibitem{Ambuhl09} Amb\"uhl C, Mastrolili M (2009) Single machine precedence constrained scheduling is a vertex cover problem. {\em Algorithmica} 53:488–-503.
	
	\bibitem{Ambuhl11} Amb\"uhl C, Mastrolilli M, Mutsanas N, Svensson O (2011) On the approximability of single-machine scheduling with precedence constraints. {\em Math. Oper. Res.} 36(4):653--669.
	
	\bibitem{Ambuhl07} Amb\"uhl C, Mastrolili M., Svensson O (2007) Inapproximability results for sparsest cut, optimal linear arrangement, and precedence constrained scheduling. In {\em Foundations of Computer Science, 2007. FOCS'07. 48th Annual IEEE Symposium on}, 329--337.
	
	\bibitem{BansalKhot} Bansal N, Khot S (2009) Optimal long code test with one free bit. In {\em Foundations of Computer Science, 2009. FOCS'09. 50th Annual IEEE Symposium on}, 453--462.
	
	\bibitem{Bansal} Bansal N, D\"urr C, Thang NK, V\'asquez \'OC (2016) The local-global conjecture for scheduling with non-linear cost. {\em Journal of Scheduling} 20(3):239--254.
	
	\bibitem{Bellman} Bellman R (1957) {\em Dynamic Programming}, Princeton University Press, Princeton, NJ
	
	
	\bibitem{Conforti} Conforti, M.,  Cornu\'ejols, G. (1984). Submodular set functions, matroids and the greedy algorithm: tight worst-case bounds and some generalizations of the Rado-Edmonds theorem. {\em Discr. Appl. Math.} 7(3):251--274.
	
	\bibitem{Chekuri} Chekuri C, Motwani R (1999) Precedence constrained scheduling to minimize sum
	of weighted completion times on a single machine. {\em Discr. Appl. Math.} 98(1):29--38.
	
	\bibitem{Chudak} Chudak FA, Hochbaum DS (1999) A half-integral linear programming relaxation for scheduling
	precedence-constrained jobs on a single machine. {\em Oper. Res. Lett.} 25:199--204.
	
	\bibitem{Correa} Correa JR, Schulz AS (2005) Single-machine scheduling with precedence constraints. {\em Mathematics of Operations Research} 30(4):1005--1021.
	
	\bibitem{Cunningham} Cunningham W (1983) Decomposition of submodular functions. {\em Combinatorica} 3(1):53--68.
	
	\bibitem{Durr} D\"{u}rr C, J\.{e}z \L, V\'asquez \'OC (2015) Scheduling under dynamic speed-scaling for minimizing weighted completion
	time and energy consumption. {\em Discrete Appl. Math.} 196:20--27.
	
	\bibitem{Edmonds} Edmonds, J (1970) Submodular functions, matroids, and certain polyhedra. In {\em Combinatorial Structures and Their Applications}. Guy R, Hanani H, Sauer N, Sch\"{o}nheim J, eds., Gordon and Breach, New York, pp.69--87.
	
	\bibitem{FKR} Fokkink R, Ramsey D, Kikuta K (2016) The search value of a set. {\em Annals Oper. Res.} 1--11.
	
	\bibitem{Fujishige} Fujishige S (2005) Submodular functions and optimization. {\em Annals Discr. Math.} 58, Elsevier.
	
	\bibitem{Gal80} Gal S (1980) Search games (Academic Press, New York).
	
	\bibitem{Gal} Gal S (2011) Search games. {\em Wiley Encylopedia of OR and MS} (John Wiley
	\& Sons, Hoboken, NJ).
	
	\bibitem{Garey} Garey MR, Johnson DS (1979) Computers and intractability: a guide to the theory of NP-completeness (Freeman, San Francisco).
	
	\bibitem{Gittins} Gittins JC, Glazebrook K, Weber RR (2011) {\em Multi-armed bandit allocation indices}, 2nd ed., Wiley, New York.
	
	\bibitem{Gluss} Gluss B (1959) An optimum policy for detecting a fault in a complex system. {\em Oper. Res.} 7(4):468--477.
	
	\bibitem{Hall} Hall LA, Schulz AS, Shmoys DB, Wein J (1997) Scheduling to minimize average completion
	time: off-line and on-line algorithms. {\em Math. Oper. Res.} 22(3):513--544.
	
	\bibitem{HeldKarp} Held M, Karp RM (1962) A dynamic programming approach to sequencing problems. {\em J. Soc. Ind. Appl. Math.} 10(1):196--210.
	
	\bibitem{Hohn} H\"{o}hn W, Jacobs T (2015) On the Performance of Smith’s Rule in Single-Machine Scheduling with Nonlinear Cost. {\em ACM T. Algorithms (TALG)} 11(4):25.
	
	\bibitem{Iwata} Iwata S, Murota K, Shigeno M (1997) A fast parametric submodular intersection algorithm for strong map sequences. {\em Math. Oper. Res.} 22(4):803--813.
	
	\bibitem{Iwata12} Iwata S, Tetali P, Tripathi P (2012). Approximating Minimum Linear Ordering Problems. In {\em Approximation, Randomization, and Combinatorial Optimization. Algorithms and Techniques}, Gupta A, Jansen K, Rolim J, Servedio R (eds), Lecture Notes in Computer Science, 7408:206--217. Springer, Berlin, Heidelberg.
	
	\bibitem{Lawler78} Lawler EL (1978) Sequencing jobs to minimize total weighted completion time {\em Ann. Discrete Math.} 2:75--90.
	
	\bibitem{Lawler} Lawler EL, Lenstra JK, Rinnooy Kan AHG, Shmoys DB (1993) Sequencing and scheduling: algorithms and complexity. {\em Handbooks in Operations Research and Managament Science 4} Elsevier, 445--522.
	
	\bibitem{Lidbetter} Lidbetter T (2013) Search games with multiple hidden objects, {\em SIAM J. Control and Optim.} 51(4):3056--3074.
	
	\bibitem{Matula} Matula D (1964) A periodic optimal search. {\em Am. Math. Mon.} 71(1):15--21.
	
	\bibitem{Margot} Margot F, Queyranne M, Wang Y (2003) Decompositions, network flows and a precedence constrained
	single machine scheduling problem. {\em Oper. Res.} 51(6):981--992.
	
	\bibitem{Megiddo} Megiddo N, Hakimi SL, Garey MR, Johnson DS, Papadimitriou CH (1988) The complexity of searching a graph. {J. ACM} 35(1):18--44.
	
	\bibitem{Megow} Megow N, Verschae J (2013) Dual techniques for scheduling on a machine with varying speed. In {\em Proc. of
		the 40th International Colloquium on Automata, Languages and Programming (ICALP)}, 745--756.
	
	\bibitem{Mitten} Mitten LG (1960) An analytic solution to the least cost testing sequence problem. {J. Ind. Eng.} 11(1):17.
	
	
	\bibitem{Momma-Sidney} Monma CL, Sidney JB (1979) Sequencing with series-parallel precedence constraints. {\em Math. Oper. Res.} 4(3):215--224.
	
	\bibitem{Neumann} von Neumann J, Morgenstern O (2007) Theory of games and economic behavior (Princeton University Press).
	
	\bibitem{Pisaruk92} Pisaruk NN (1992) The boundaries of submodular functions. Comp. Math. Math. Phys 32(12):1769--1783.
	
	\bibitem{Pisaruk} Pisaruk NN (2003) A fully combinatorial 2-approximation algorithm for precedence-constrained scheduling
	a single machine to minimize average weighted completion time. {\em Discrete Appl. Math.} 131(3):655--663.
	
	\bibitem{Queyranne} Queyranne M (1998) Minimizing symmetric submodular functions. {\em Mathematical Programming} 82(1--2):3--12.
	
	\bibitem{Rothkopf} Rothkopf MH (1966). Scheduling independent tasks on parallel processors. {\em Management Science} 12(5):437--447.
	
	\bibitem{Schrijver} Schrijver A (2003) Combinatorial optimization, polyhedra and efficiency, {\em Algorithms and Combinatorics 24} (Springer, Berlin).
	
	\bibitem{Schulz} Schulz AS (1996) Scheduling to minimize total weighted completion time: performance guarantees of
	LP-based heuristics and lower bounds. In {\em Proceedings of the 5th Conference on Integer Programming and Combinatorial Optimization (IPCO)}, 301--315.
	
	\bibitem{Schulz-Verschae} Schulz AS, Verschae J (2016) Min-sum scheduling under precedence constraints. In {\em European Symposia on Algorithms, 2016 (ESA 2016)}, forthcoming.
	
	
	
	\bibitem{Sidney} Sidney JB (1975) Decomposition algorithms for single-machine sequencing with precedence relations and
	deferral costs. {\em Oper. Res.} 23(2):283--298.
	
	\bibitem{Sidney-Steiner} Sidney JB, Steiner G (1986) Optimal sequencing by modular decomposition: Polynomial algorithms. {\em Oper. Res} 34(4):606--612.
	
	\bibitem{Smith} Smith WE (1956) Various optimizers for single‐stage production. {\em Naval Res. Logist. Quarterly} 3(1-‐2):59--66.
	
	\bibitem{Stengel} von Stengel B, Werchner R (1997) Complexity of searching an immobile hider in a graph. {\em Discrete Appl. Math.} 78(1):235--249.
	
	
	\bibitem{Stiller} Stiller S, Wiese A (2010) Increasing Speed Scheduling and Flow Scheduling. In {\em Algorithms and Computation (ISAAC)}, 279–290.
	
	\bibitem{Vocking} V\"ocking B (2007) Selfish load balancing. In {\em Algorithmic Game Theory},
	eds. Nisam N, Roughgarden T, Tardos \'E, Vazirani VV (Cambridge University Press),
	517--542.
	
	\bibitem{Vondrak} Vondr\'{a}k J (2010) Submodularity and curvature: the optimal algorithm. {\em RIMS Kokyuroku Bessatsu B} 23:253--266.
	
\end{thebibliography}
\end{document}